\documentclass[12pt]{amsart}
\usepackage{amssymb}
\usepackage{amsfonts}
\usepackage{color}
\usepackage{hyperref}
\usepackage{blindtext}
\usepackage{mathtools}
\usepackage{bbm}
\allowdisplaybreaks

\theoremstyle{definition}
\newtheorem{theorem}{Theorem}[section]

\newtheorem{lemma}[theorem]{Lemma}
\newtheorem{proposition}[theorem]{Proposition}
\newtheorem{example}[theorem]{Example}

\newtheorem{definition}[theorem]{Definition}

\newcommand{\mb}{\mathbb}
\newcommand{\mc}{\mathcal}
\newcommand{\eul}{\mathfrak}

\newcommand{\bou}{_{\scriptscriptstyle{\rm b}}}

\newcommand{\A}{\eul A}
\newcommand{\Ao}{{\eul A}_{\scriptscriptstyle 0}}

\newcommand{\Bo}{{\eul B}_{\scriptscriptstyle 0}}

\newcommand{\vp}{\varphi}

\newcommand{\Hil}{{\mc H}}

\newcommand{\D}{{\mc D}}

\def\x{\relax\ifmmode {\mbox{*}}\else*\fi}

\newcommand{\id}{\mathbbm{1}}
\newcommand{\ip}[2]{\left\langle{#1}|{#2}\right\rangle}

\newcommand{\ad}{^{\mbox{\scriptsize $\dag$}}}
\newcommand{\LDH}{{\mathcal L}\ad(\D,\Hil)}
\newcommand{\LDHpi}{{\mathcal L}\ad(\D_{\scriptscriptstyle\pi},\Hil_{\scriptscriptstyle\pi})}

\newcommand{\SSA}{{\mathcal S}_{\Ao}(\A)}

\newcommand{\SSHone}{{\mathcal S}_{\Ao}(\H_1)}
\newcommand{\SSHtwo}{{\mathcal S}_{\Bo}(\H_2)}

\newcommand{\rep}{{\mc R}(\A,\Ao)}
\newcommand{\crep}{{\mc R}_c(\A,\Ao)}

\newcommand{\normo}[1]{\|#1\|_{\scriptscriptstyle0}}
\def\H{{\mathcal H}}
\newcommand{\wmult}{\mbox{\raisebox{1pt}{$\scriptscriptstyle{
\square}$}}}

\numberwithin{equation}{section}

\numberwithin{equation}{section}

\begin{document}

\title[]{On some applications of representable and continuous functionals of Banach quasi *-algebras}
\author{Maria Stella Adamo}
\address{Dipartimento di Matematica, Universit\`a di Roma ``Tor Vergata'', I-00174 Roma, Italy}
\email{adamo@axp.uniroma2.it; msadamo@unict.it}
\keywords{Representable functionals, Banach and Hilbert quasi *-algebras, Weak derivations on Banach quasi *-algebras, Tensor product of Hilbert quasi *-algebras}
\subjclass[2010]{Primary 46L08; Secondary 46A32, 46L57, 46L89, 47L60}

\begin{abstract} This survey aims to highlight some of the consequences that representable (and continuous) functionals have in the framework of Banach quasi *-algebras. In particular, we look at the link between the notions of *-semisimplicity and full representability in which representable functionals are involved. Then, we emphasize their essential role in studying *-derivations and representability properties for the tensor product of Hilbert quasi *-algebras, a special class of Banach quasi *-algebras.
\end{abstract}

\maketitle

\section{Introduction and preliminaries}

The investigation of (locally convex) quasi *-algebras was undertaken around the beginning of the '80s, in the last century, to give a solution to specific problems concerning quantum statistical mechanics and quantum field theory, that required instead a representation of observables as \textit{unbounded} operators, see, e.g., \cite{Bag6,ct2}. They were introduced by G. Lassner in the series of papers \cite{las} and \cite{las1} in 1988.

A particular interest has been shown for the theory of *-rep\-re\-sen\-ta\-tions of quasi *-algebras in a specific family of unbounded densely defined and closable operators. In this framework, a central role is played by \textit{representable functionals}, i.e., those functionals that admit a GNS-like construction. In the process of looking at structural properties of locally convex quasi *-algebras, *-semisimplicity and full representability are the critical notions involved, making an extensive use of representable (and continuous) functionals \cite{AT,Ant1,Frag2,Trap1}.

The goal of this survey is to point out some of the various connections that representable and continuous functionals for Banach quasi *-algebras have and examine their applications, e.g., the study of unbounded *-derivations and tensor products. The importance of investigating these themes stem from the study of physical phenomena. Moreover, very little is known about unbounded *-derivations and tensor products in the framework of unbounded operator algebras. The reader is referred to as \cite{fiw,fiw1,hei,WZ1,WZ2,WZ3}.

Banach quasi *-algebras constitute a particular subclass of locally convex *-algebras. In this context, the interesting question is to understand whether representable functionals are continuous. The reasons to get to know about the continuity of these functionals are several. Among them, the continuity is a crucial feature for representable functionals, since it would reflect on the sesquilinear forms and the *-representations associated with them through the GNS-like construction (see \cite{Trap1}). Furthermore, no example of representable functional that is \textit{not continuous} is known in the literature. 

A positive answer to this question has been given for the space $L^2(I,d\lambda)$, where $I=[0,1]$ and $\lambda$ is the Lebesgue measure, over continuous or essentially bounded functions, and more in general, for commutative \textit{Hilbert} quasi *-algebras under certain conditions. In the case of $L^p$-spaces for $p\geq1$, we observe a discontinuous behaviour in the quantity of representable and continuous functional when $p\geq1$ gets bigger. For $p\geq2$, the $L^p$-spaces turn out to be \textit{fully representable} and \textit{*-semisimple} Banach quasi *-algebras and in this case these notions \textit{coincide} (see Section \ref{Sec2}, \cite{AT,Ant1,Bag5,Bag7,Frag2}). 

Hilbert quasi *-algebras constitute a class of *-semisimple and fully representable quasi *-algebras. In this case, representable and continuous functionals are in 1-1 correspondence with weakly positive and bounded elements. This result allows us to get the representability for the tensor product of two representable and continuous functionals in the framework of tensor product of Hilbert quasi *-algebras (refer to \cite{adamo,AF}).

*-Semisimplicity allows us to define a proper notion of *-derivation in the case of Banach quasi *-algebras and prove a result characterizing infinitesimal generators of one-parameter group of *-automorphisms in the Banach quasi *-algebras case, extending results of Bratteli - Robinson for C*-algebras (see, e.g., \cite{Brat2}).

The survey is structured as follows. Firstly, we give some preliminaries about (Banach) quasi *-algebras, representable (and continuous) functionals and the GNS-construction. For these quasi *-algebras, in Section \ref{Sec2} we recall the notions of *-semisimplicity and full representability, summing up the main results about their link in the Banach case and the characterization we have for Hilbert quasi *-algebras. In Section \ref{Sec3}, we concentrate on the case of *-semisimple Banach quasi *-algebras and deal with weak *-derivations and related results. In the last section, we look at the construction of the tensor product Hilbert quasi *-algebra and look at the properties of tensor products of representable functionals and the sesquilinear forms involved in the definition of *-semisimplicity.

For the reader's convenience, we recall some preliminary notions for future use. Further details can be found in \cite{Ant1}.

\begin{definition}
A \textit{quasi *-algebra} $(\A,\Ao)$ (or over $\Ao$) is a pair consisting of a vector space $\A$ and a *-algebra $\Ao$ contained in $\A$ as a subspace and such that
\begin{itemize}
\item[(i)] the left multiplication $ax$ and the right multiplication $xa$ of an element $a\in\A$ and $x\in\Ao$ are always defined and bilinear;
\item[(ii)] $(xa)y=x(ay)$ and $a(xy)=(ax)y$ for each $x,y\in\Ao$ and $a\in\A$;
\item[(iii)] an involution $\ast$, which extends the involution of $\Ao$, is defined in $\A$ with the property $(ax)^*=x^*a^*$ and $(xa)^*=a^*x^*$ for all $a\in\A$ and $x\in\Ao$.
\end{itemize}
\end{definition}

We say that a quasi *-algebra $(\A,\Ao)$ has a \textit{unit}, if there is an element in $\Ao$, denoted by $\id$, such that $a\id=a=\id a$ for all $a\in\A$. If a unit exists, then it is always unique.

\begin{definition}\label{fun_repr}
	Let $(\A,\Ao)$ be a quasi *-algebra. A linear functional $\omega:\A\to\mathbb{C}$ satisfying
	\begin{itemize}
		\item[(L.1)] $\omega(x^*x)\geq0$ for all $x\in\Ao$;
		\item[(L.2)] $\omega(y^*a^*x)=\overline{\omega(x^*ay)}$ for all $x,y\in\Ao$, $a\in\A$;
			\item[(L.3)] for all $a\in\A$, there exists $\gamma_a>0$ such that
			$$|\omega(a^*x)|\leq\gamma_a\omega(x^*x)^{\frac12},\quad\forall x\in\Ao,$$
	\end{itemize}
is called \textit{representable} on the quasi *-algebra $(\A,\Ao)$.

The family of all representable functionals on the quasi *-algebra $(\A,\Ao)$ will be denoted by $\mathcal{R}(\A,\Ao)$.
\end{definition}

This definition is justified by the following Theorem \ref{thm_repr}, proving the existence of a GNS-like construction of a *-representation $\pi_\omega$ and a Hilbert space $\mathcal{H}_{\omega}$ for a representable functional $\omega$ on $\A$. 
\smallskip

Let $\Hil$ be a Hilbert space and let $\D$ be a dense linear subspace of $\Hil$. We denote by $\LDH$ the following family of closable operators:
$$\LDH=\left\{X:\D\to\Hil:\mathcal{D}(X)=\D, \mathcal{D}(X^{\ast})\supset\D\right\}.$$
$\LDH$ is a $\mathbb{C}-$vector space with the usual sum and scalar multiplication. If we define the involution $\dagger$ and partial multiplication $\wmult$ as
$$X\mapsto X^{\dagger}\equiv X^{\ast}\upharpoonright_{\D}\quad\text{and}\quad X\wmult Y=X^{\dagger\ast}Y,$$
then $\LDH$ is a partial *-algebra defined in \cite{Ant1}.
\smallskip

\begin{definition}
	A {\it *-representation} of a quasi *-algebra $(\A,\Ao)$ is a *-homomorphism $\pi:\A\to\LDHpi$, where $\D_{\scriptscriptstyle\pi}$ is a dense subspace of the Hilbert space $\Hil_{\scriptscriptstyle\pi}$, with the following properties:
	\begin{itemize}
		\item[(i)] $\pi(a^{\ast})=\pi(a)^{\dagger}$ for all $a\in\A$;
		\item[(ii)] if $a\in\A$ and $x\in\Ao$, then $\pi(a)$ is a left multiplier of $\pi(x)$ and $\pi(a)\wmult\pi(x)=\pi(ax)$.
	\end{itemize}
\end{definition}

A *-representation $\pi$ is 
\begin{itemize}
	\item {\em cyclic} if $\pi(\Ao)\xi$ is dense in $\Hil_{\scriptscriptstyle\pi}$ for some $\xi\in\D_{\scriptscriptstyle\pi}$;
	\item {\em closed} if $\pi$ coincides with its closure $\widetilde{\pi}$ defined in \cite[Section 2]{Trap1}.
\end{itemize}
If $(\A,\Ao)$ has a unit $\id$, then we suppose that $\pi(\id)=I_\D$, the identity operator of $\D$.
\smallskip

\begin{theorem}\cite{Trap1}\label{thm_repr}
	Let $(\A, \Ao)$ be a quasi *-algebra with unit
	$\id$ and let $\omega$ be a linear functional on $(\A, \Ao)$ that satisfies the conditions (L.1)-(L.3) of Definition \ref{fun_repr}. Then, there exists a closed cyclic *-representation ${\pi}_\omega$ of $(\A,\Ao)$, with cyclic vector $\xi_\omega$ such that
	$$ \omega(a)=\ip{{\pi}_\omega(a)\xi_\omega}{\xi_\omega},\quad \forall a\in \A.$$
	This representation is unique up to unitary equivalence.
\end{theorem}

\subsection{Normed quasi *-algebras}

\begin{definition}
A  quasi *-algebra $(\A,\Ao)$ is called a {\it normed quasi *-algebra} if a norm
$\|\cdot\|$ is defined on $\A$ with the properties
\begin{itemize}
	\item[(i)]$\|a^*\|=\|a\|, \quad \forall a \in \A$;
	\item[(ii)] $\Ao$ is dense in $\A$;
	\item[(iii)]for every $x \in \Ao$, the map $R_x: a \in \A \to ax \in \A$ is continuous in
	$\A$.
\end{itemize}
\end{definition}
The continuity of the involution implies that
	\begin{itemize}
		\item[(iii')]for every $x \in \Ao$, the map $L_x: a \in \A \to xa \in \A$ is continuous in
		$\A$.
	\end{itemize}

\begin{definition}
	If $(\A,\| \cdot \|) $ is a Banach space, we say that $(\A,\Ao)$ is a {\it Banach quasi *-algebra}.
\end{definition}
The norm topology of $\A$ will be denoted by $\tau_n$. 
\smallskip

An important class of Banach quasi *-algebras is given by \textit{Hilbert quasi *-algebras}. 

\begin{definition}\label{Hilb_qalg}
	Let $\Ao$ be a *-algebra which is also a pre-Hilbert space with respect to the inner product $\ip{\cdot}{\cdot}$ such that:
	\begin{enumerate}
		\item the map $y\mapsto xy$ is continuous with respect to the norm defined by the inner product;
		\item $\ip{xy}{z}=\ip{y}{x^*z}$ for all $x,y,z\in\Ao$;
		\item $\ip{x}{y}=\ip{y^*}{x^*}$ for all $x,y\in\Ao$;
		\item $\Ao^2$ is total in $\Ao$.
	\end{enumerate}
Such a *-algebra $\Ao$ is said to be a \textit{Hilbert algebra}. If $\H$ denotes the Hilbert space completion of $\Ao$ with respect to the inner product $\ip{\cdot}{\cdot}$, then $(\H,\Ao)$ is called {\it Hilbert quasi *-algebra}.
\end{definition}

\section{Analogy between *-semisimplicity and full representability}\label{Sec2}
Representable functionals constitute a valid tool for investigating structural properties of Banach quasi *-algebras. Indeed, the notions of *-semisimplicity and full representability are strongly related to these functionals. They turn out to be the same notion when dealing with quasi *-algebras that verify the condition $(P)$. For further reading, see \cite{AT, Ant1,Bag5,Frag2,Trap1}.

If $\omega \in \rep$, then we can associate with it the sesquilinear form $\varphi_{\omega}$ defined on $\Ao\times\Ao$ as
\begin{equation} \label{sesqu_ass} \vp_\omega(x,y):= \omega(y^*x), \quad x,y \in \Ao.\end{equation}

If $(\A, \Ao)$ is a normed quasi *-algebra, we denote by $\crep$ the subset of $\rep$ consisting of continuous functionals.

Let $\omega\in\mathcal{R}_c(\A,\Ao)$. Then $\omega$ is continuous on $\A$, but $\varphi_{\omega}$ is not necessarily continuous on $\Ao\times\Ao$. $\varphi_{\omega}$ is said to be \textit{closable} if for every sequence of elements $\{x_n\}$ in $\Ao$ such that
\begin{equation}\label{clos}x_n\to 0\quad\text{and}\quad\varphi(x_n-x_m,x_n-x_m)\to0\quad\text{as}\quad n\to\infty\end{equation}
then $\varphi(x_n,x_n)\to0$ as $n\to\infty$. 

By the condition \eqref{clos}, the closure of $\varphi_{\omega}$, denoted by $\overline{\varphi}_{\omega}$, is a well-defined sesquilinear form on $\mathcal{D}(\overline{\varphi}_{\omega})\times\mathcal{D}(\overline{\varphi}_{\omega})$ as
$$\overline{\varphi}_{\omega}(a,a):=\lim_{n\to\infty}\varphi_{\omega}(x_n,x_n),$$
where $\mathcal{D}(\overline{\varphi}_{\omega})$ is the following dense domain
\begin{multline}\notag \D(\overline{\varphi}_{\omega})=\{a\in\A:\exists\{x_n\}\subset\Ao\;\text{s.t.}\;x_n\to a\;\text{and}\;\\
\varphi_{\omega}(x_n-x_m,x_n-x_m)\to0\}.
\end{multline}

For a locally convex quasi *-algebra $(\A,\Ao)$, $\overline{\varphi}_{\omega}$ always exists, \cite{Frag2}. Nevertheless, it is unclear whether $\mathcal{D}(\overline{\varphi}_{\omega})$ is the whole space $\A$. We show in Proposition \ref{prop1} below that $\mathcal{D}(\overline{\varphi}_{\omega})$ is $\A$ in the case of a Banach quasi *-algebra.

\begin{proposition}\cite{AT}\label{prop1}
	Let $(\A, \Ao)$ be a Banach quasi *-algebra with unit $\id$,  $\omega \in \crep$ and $\vp_{\omega}$ the associated sesquilinear form on $\Ao \times \Ao$ defined as in \eqref{sesqu_ass}. Then $\D(\overline{\vp}_\omega)=\A$;  hence $\overline{\vp}_\omega$ is everywhere defined and bounded.
\end{proposition}

Set now
$$\Ao^+:=\left\{\sum_{k=1}^n x_k^* x_k, \, x_k \in \Ao,\, n \in {\mb N}\right\}.$$
Then $\Ao^+$ is a wedge in $\Ao$ and we call the elements of $\Ao^+$ \emph{positive elements of} $\Ao$.
As in \cite{Frag2}, we call {positive elements of} $\A$ the members of $\overline{\Ao^+}^{\tau_n}$. We set $\A^+:=\overline{\Ao^+}^{\tau_n}$.

\begin{definition} A linear functional on $\A$ is \textit{positive} if $\omega(a)\geq0$ for every $a\in\A^+$. A family of positive linear functionals $\mc F$ on
	$(\A[\tau_n], \Ao)$ is called {\it sufficient} if for every $a \in
	\A^+$, $a \neq 0$, there exists $\omega \in {\mc F}$ such that $\omega
	(a)>0$.
\end{definition}

\begin{definition}\label{fully_rep} A normed quasi $^{\ast}$-algebra
	$(\A[\tau_n],\Ao)$ is called {\it fully representable} if ${\mc
		R}_c(\A,\Ao)$ is sufficient and $\D(\overline{\varphi}_{\omega})=\A$ for every $\omega$ in $\mathcal{R}_c(\A,\Ao)$.
\end{definition}

We denote by $\mathcal{Q}_{\Ao}(\A)$ the family of all sesquilinear forms $\Omega:\A\times\A\to\mathbb{C}$ such that
\begin{itemize}
	\item[(i)] $\Omega(a,a)\geq0$ for every $a\in\A$;
	\item[(ii)] $\Omega(ax,y)=\Omega(x,a^*y)$ for every $a\in\A$, $x,y\in\Ao$;
\end{itemize}
We denote by $\SSA$ the subset of $\mathcal{Q}_{\Ao}(\A)$ consisting of all continuous sesquilinear forms having the property that also
\begin{itemize}
	\item[(iii)]$ |\Omega(a,b)|\leq \|a\|\|b\|$, for all $a, b\in \A.$
\end{itemize}

\begin{definition}\label{def1}
	A normed quasi *-algebra $(\mathfrak{A}[\tau_n],\Ao)$ is called {\bf *-semi\-simple} if, for every $0\neq a\in\A$, there exists $\Omega\in\mathcal{S}_{\Ao}(\A)$ such that $\Omega(a,a)>0$.
\end{definition}

Proposition \ref{prop1} is useful to show the following result, making clear what the link is between *-semisimplicity and full representability. We need first to introduce the following condition of positivity $(P)$
\begin{equation}\notag\label{P} a\in\A\;\text{and}\;\omega(x^*ax)\geq0\;\;\forall\omega\in\mathcal{R}_c(\A,\Ao)\;\text{and}\;x\in\Ao\;\;\Rightarrow\;\;a\in\A^+.
\end{equation}

\begin{theorem}\cite{AT}\label{thm_fullrep_semis} Let $(\A, \Ao)$ be a Banach quasi *-algebra with unit $\id$. The following statements are equivalent. 
	\begin{itemize}
		\item[(i)]$\crep$ is sufficient.
		\item[(ii)]$(\A,\Ao)$ is fully representable.
	\end{itemize}
	If the condition $(P)$ holds, (i) and  (ii) are equivalent to the following
	\begin{itemize}
		\item[(iii)]$(\A,\Ao)$ is *-semisimple.
	\end{itemize}
\end{theorem}

The condition $(P)$ is not needed to show (iii) $\Rightarrow$ (ii) of Theorem \ref{thm_fullrep_semis}. 

Theorem \ref{thm_fullrep_semis} shows the deep connection between full representability and *-semisimplicity for a Banach quasi *-algebra. Under the condition of positivity $(P)$, the families of sesquilinear forms involved in the definitions of full representability and *-semisimplicity can be identified. 
\smallskip

For a Hilbert quasi *-algebra $(\H,\Ao)$, representable and continuous functionals are in 1-1 correspondence with a certain family of elements in $\H$. 

\begin{definition}\label{Hilb_wp_bou}
	Let $(\H,\Ao)$ be a Hilbert quasi *-algebra. An element $\xi\in\H$ is called
	\begin{itemize}
		\item[(i)] {\bf weakly positive} if the operator $L_{\xi}:\Ao\to\H$ defined as $L_{\xi}(x)=\xi x$ is positive.
		\item[(ii)] {\bf bounded} if the operator $L_{\xi}:\Ao\to\H$ is bounded.
	\end{itemize}
	The set of all weakly positive (resp. bounded) elements will be denoted as $\H^+_w$ (resp. $\H_{\bou}$), see \cite{AT,ct1}.
\end{definition}

\begin{theorem}\cite{AT}\label{Hrepc}
	Let $(\H,\Ao)$ be a Hilbert quasi *-algebra. Then $\omega\in\mathcal{R}_c(\H,\Ao)$ if, and only if, there exists a unique weakly positive bounded element $\eta\in\H$ such that
	$$\omega(\xi)=\ip{\xi}{\eta},\quad\forall\xi\in\H.$$
\end{theorem}

\subsection{Case of $L^p(I,d\lambda)$ for $p\geq1$} Consider $\Ao$ to be $L^{\infty}(I,d\lambda)$, where $I$ is a compact interval of the real line and $\lambda$ is the Lebesgue measure. Let $\tau_n$ be the topology generated by the $p$-norm $\|\cdot\|_p$
$$\|f\|_p:=\left(\int_I|f|^pd\lambda\right)^{\frac1p},\quad\forall f\in L^{\infty}(I,d\lambda),$$
for $p\geq1$. Then, the completion of $L^{\infty}(I,d\lambda)$ with respect to the $\|\cdot\|_p$-norm is given by $L^p(I,d\lambda)$.

We conclude that, for $p\geq 1$, $(L^p(I,d\lambda),L^{\infty}(I,d\lambda))$ is a Banach quasi *-algebra. The same conclusion holds if we consider the *-algebra of continuous functions over $I=[0,1]$, denoted by $\mathcal{C}(I)$.

For $p\geq 2$, $(L^p(I,d\lambda),L^{\infty}(I,d\lambda))$ and $(L^p(I,d\lambda),\mathcal{C}(I))$ are fully representable and *-semisimple Banach quasi *-algebras. 

For $1\leq p<2$, we have $\mathcal{R}_c(L^p(I,d\lambda),\Ao)=\{0\}$ for both $\Ao=L^{\infty}(I,d\lambda)$ or $\Ao=\mathcal{C}(I)$. 

Note that the Banach quasi *-algebras $(L^p(I,d\lambda),L^{\infty}(I,d\lambda))$ and $(L^p(I,d\lambda),\mathcal{C}(I))$ verify the condition $(P)$ for all $p\geq1$.

The absence of representable and continuous functionals that we observe passing the threshold $p=2$ and the lack in the literature of an example of representable functional that is not continuous led us to pose the following question
\smallskip

\textbf{Question}: Is every representable functional on a Banach quasi *-algebra continuous?
\smallskip

The answer is positive in the case of the Hilbert quasi *-algebras $(L^2(I,d\lambda),L^{\infty}(I,d\lambda))$ and $(L^2(I,d\lambda),\mathcal{C}(I))$, see \cite{AT}.

\begin{theorem}\cite{AT}
	Let $\omega$ be a representable functional on $(L^2(I,d\lambda), \allowbreak L^{\infty}(I,d\lambda))$. Then there exists a {unique} bounded finitely additive measure $\nu$ on $I$ which vanish on subsets of $I$ of zero $\lambda-$measure and a unique bounded linear operator $S: L^2(I, d\lambda)\to L^2(I, d\nu)$ such that
	\begin{equation}\notag \omega(f) = \int_I (Sf)d\nu, \quad \forall f \in L^2(I,d\lambda).\end{equation}
	The operator $S$ satisfies the following conditions:
	\begin{align}
	&S(f\phi)=(Sf)\phi=\phi(Sf)\quad \forall f\in L^2(I,d\lambda),\phi\in L^{\infty}(I,d\lambda); \\
	&S\phi=\phi, \quad \forall \phi \in L^{\infty}(I,d\lambda). 
	\end{align}
	Thus, every representable functional $\omega$ on $\left(L^2(I,d\lambda),L^{\infty}(I,d\lambda)\right)$ is continuous.
\end{theorem}

\begin{theorem}\cite{AT}
Let $\omega$ be a representable functional on $(L^2(I,d\lambda), \allowbreak\mathcal{C}(I))$. Then there exists a {unique} Borel measure $\mu$ on $I$ and a {unique} bounded linear operator $T: L^2(I, d\lambda)\to L^2(I, d\mu)$ such that
\begin{equation}\notag \omega(f) = \int_I (Tf)d\mu, \quad \forall f \in L^2(I,d\lambda).\end{equation}
	The operator $T$ satisfies the following conditions:
		\begin{align}
		&T(f\phi)=(Tf)\phi=\phi(Tf)\quad \forall f\in L^2(I,d\lambda),\phi\in \mathcal{C}(I); \\
		&T\phi=\phi, \quad \forall \phi \in \mathcal{C}(I). 
		\end{align}
		Thus, every representable functional $\omega$ on $\left(L^2(I,d\lambda),\mathcal{C}(I)\right)$ is continuous.
		Moreover, $\mu$ is absolutely continuous with respect to $\lambda$.
\end{theorem}

The above theorems can be extended to the case of a commutative Hilbert quasi *-algebra, under certain hypotheses.

	\begin{theorem}\cite{AT}
		Let $(\mathcal{H},\Ao)$ be a commutative Hilbert quasi *-algebra with unit $\id$.
		Assume that $\Ao[\normo{\cdot}]$ is a Banach *-algebra and that there exists an element $x$ of $\Ao$ such that the spectrum $\sigma(\overline{R}_x)$ of the bounded operator $\overline{R}_x$ of right multiplication by $x$ consists only of its continuous part $\sigma_c(\overline{R}_x)$.
		If $\omega$ is representable on $(\A,\Ao)$, then $\omega$ is bounded.
	\end{theorem}

The main tools to get these results are the Riesz-Markov Representation Theorem for continuous functions on a compact space and the intertwining theory on Hilbert spaces (see \cite{sin}). Unfortunately, these tools turn out to be unsuitable to the case of a (non-commutative) Hilbert quasi *-algebra or a Banach quasi *-algebra.

\section{1st Application: Derivations and their closability}\label{Sec3}
*-Derivations have been widely employed to describe the dynamics for a quantum phenomena. For a quantum system of finite volume $V$, the hamiltonian belongs to the local C*-algebra and implements the inner *-derivation for the dynamics. Nevertheless, the termodynamical limit in general fails to exist, see \cite{Ant1}.

Under some assumptions, the limit turns out to be a \textit{weak *-derivation} generating a one parameter group of \textit{weak *-automorphisms}, defined for a *-semisimple Banach quasi *-algebra, as we will investigate in the following section. For detailed discussion, see \cite{AT2,Ant2,HP}. 
\medskip

If we consider a derivation $\delta:\Ao[\|\cdot\|]\to\A[\|\cdot\|]$, then this derivation is densely defined. If $\delta$ is closable, then its closure $\overline{\delta}$ as a linear map is \textit{not} a derivation in general.
\smallskip

$\bullet$ In this section, we only consider *-semisimple Banach quasi *-algebras, if not otherwise specified.
\smallskip

For these Banach quasi *-algebras, define a weaker multiplication in $\A$ as in \cite{ct1}. 
\smallskip
	
\begin{definition}
Let $(\A,\Ao)$ be a Banach quasi *-algebra. Let $a,b\in\A$. We say that the {\it weak multiplication $a\wmult b$} is well-defined if there exists a (necessarily unique) $c\in\A$ such that:
$$
\varphi(bx,a^*y)=\varphi(cx,y),\;\forall\, x,y\in\Ao, \forall\,\varphi\in\SSA.
$$
In this case, we put $a\wmult b:=c$.
\end{definition}

Let $a\in\A$. The space of right (resp. left) weak multipliers of $a$, i.e., the space of all the elements $b\in\A$ such that $a\wmult b$ (resp. $b\wmult a$) is well-defined, will be denoted as $R_w(a)$ (resp. $L_w(a)$). We indicate by $R_w(\A)$ (resp. $L_w(\A)$) the space of universal right (resp. left) multipliers of $\A$, i.e., all the elements $b\in\A$ such that $b\in R_w(a)$ (resp. $b\in L_w(a)$) for every $a\in\A$. Clearly, $\Ao\subseteq R_w(\A)\cap L_w(\A)$.
\smallskip

Let $(\A,\Ao)$ be a Banach quasi *-algebra. For every $a\in\A$, two linear operators $L_a$ and $R_a$ are defined is the following way
\begin{align}
&\label{11}L_a:\Ao\to\A\;\;L_a(x)=ax\;\;\forall x\in\Ao\\
&\label{12}R_a:\Ao\to\A\;\;R_a(x)=xa\;\;\forall x\in\Ao.
\end{align}

An element $a\in\A$ is called \textit{bounded} if the operators $L_a$ and $R_a$ defined in \eqref{11} and \eqref{12} are $\|\cdot\|$-continuous and thus extendible to the whole space $\A$. As for Hilbert quasi *-algebras in Definition \ref{Hilb_wp_bou}, the set of all bounded elements will be denoted by $\A_{\bou}$.

\begin{lemma}\cite[Lemma 2.16]{AT2}
If $(\A,\Ao)$ is a *-semisimple Banach quasi *-algebra with unit $\id$, the set $\A_{\bou}$ of bounded elements coincides with the set $R_w(\A)\cap L_w(\A)$.	
\end{lemma}

Let $\theta:\A\to \A$ a linear bijection. We say that $\theta$ is a {\it weak *-automorphism of $(\A,\Ao)$} if
\begin{itemize}
	\item[(i)] $\theta(a^*)=\theta(a)^*$, for every $a \in \A$;
	\item[(ii)] $\theta(a)\wmult \theta(b)$ is well defined if, and only if, $a\wmult b$ is well defined and, in this case, $$\theta(a\wmult b)=\theta(a)\wmult \theta(b).$$
\end{itemize}

\begin{definition}Let $\beta_t$ be a weak *-automorphism of $\A$ for every $t\in\mathbb{R}$. If
\begin{itemize} 
	\item[(i)]$\beta_0(a)=a,$ $\forall a\in \A$  
	\item[(ii)] $\beta_{t+s}(a)= \beta_t(\beta_s(a))$, $\forall a\in \A$ 
\end{itemize}
then we say that $\beta_t$ is a {\em one-parameter group of weak *-automorphisms of $(\A,\Ao)$}.
If $\tau$ is a topology on $\A$ and the map $t\mapsto \beta_t(a)$ is $\tau$-continuous, for every $a\in \A$, we say that $\beta_t$ is a {\it $\tau$-continuous} weak *-automorphism group.
\end{definition}

In this case
$$ \mathcal{D}(\delta_\tau)=\left\{a\in \A: \lim_{t\to 0} \frac{\beta_t(a)-a}{t} \mbox{ exists in $\A[\tau]$}\right\}$$
and
$$ \delta_\tau (a)=\tau-\lim_{t\to 0} \frac{\beta_t(a)-a}{t}, \quad a \in  \mathcal{D}(\delta_\tau).$$
Then $\delta_\tau$ is called the \textit{infinitesimal generator} of the one-parameter group $\{\beta_t\}$ of weak *-automorphisms of $(\A,\Ao)$.

\begin{definition}\cite{AT2}\label{defn_deriv}
	Let $(\A,\Ao)$ be a *-semisimple Banach quasi *-algebra and $\delta$ a linear map of $\mathcal{D}(\delta)$ into $\A$, where $\mathcal{D}(\delta)$ is a partial *-algebra with respect to the weak multiplication $\wmult$. We say that $\delta$ is a {\em  weak *-derivation} of $(\A,\Ao)$ if
	\begin{itemize}
		\item[(i)] $\Ao\subset\mathcal{D}(\delta)$
		\item[(ii)]$\delta(x^*)=\delta(x)^*, \; \forall x \in \Ao$
		\item[(iii)] if $a,b\in\mathcal{D}(\delta)$ and $a\wmult b$ is well defined, then $a\wmult b\in\mathcal{D}(\delta)$ and
		$$\vp(\delta(a\wmult b)x,y)= \vp(bx,\delta(a)^*y)+\vp(\delta(b)x,a^*y),$$
		for all $\vp\in \SSA$, for every $x,y \in \Ao$.
	\end{itemize}
\end{definition}

Analogously to the case of a C*-algebra, to a \textit{uniformly bounded} norm continuous weak *-automorphisms group there corresponds a closed weak *-derivation that generates the group (see \cite{Brat2}). 

\begin{theorem}\cite{AT2}\label{HY1} Let $\delta:\mathcal{D}(\delta)\to\A[\|\cdot\|]$ be a weak *-derivation on a *-semisimple Banach quasi *-algebra $(\A,\Ao)$. Suppose that $\delta$ is the infinitesimal generator of a uniformly bounded, {$\tau_n$-continuous} group of weak *-automorphisms of $(\A,\Ao)$. Then $\delta$ is closed; its resolvent set $\rho(\delta)$ contains {$\mathbb{R}\setminus\{0\}$} and
	\begin{equation}\label{eqn_lowbound}\|\delta(a)-\lambda a\|\geq|\lambda|\,\|a\|,\quad a\in\mathcal{D}(\delta),\lambda \in {\mb R}.\end{equation}
\end{theorem}

\begin{theorem}\cite{AT2}\label{HY2} Let $\delta:\mathcal{D}(\delta)\subset\A_{\bou}\to\A[\|\cdot\|]$ be a closed weak *-derivation on a *-semisimple Banach quasi *-algebra $(\A,\Ao)$. Suppose that $\delta$ verifies the same conditions on its spectrum of Theorem \ref{HY1} and $\Ao$ is a core for every multiplication operator $\hat{L_a}$ for $a\in\A$, i.e. $\hat{L}_a=\overline{L}_a$. Then $\delta$ is the infinitesimal generator of a uniformly bounded, {$\tau_n$-continuous} group of weak *-automorphisms of $(\A,\Ao)$.
\end{theorem}

In Theorem \ref{HY2}, we assumed further conditions, for instance that the domain $\mathcal{D}(\delta)$ of the weak *-derivation is contained in $\A_{\bou}$, which turns out to be satisfied in some interesting situations such as the weak derivative in $L^p-$spaces.

\begin{example}[Inner weak *-derivations for unbounded hamiltonian $h$] Let $(\A,\Ao)$ be a *-semisimple Banach quasi *-algebra. Let $h\in\A$ be a self-adjoint unbounded element, i.e., $h=h^*$ and $\sigma(h)\subset\mathbb{R}$. We define the following derivation
$$\delta_h:\Ao\to\A,\quad\text{defined as}\quad\delta_h(x)=i[h,x],\quad x\in\Ao.$$
Then, for every fixed $t\in\mathbb{R}$, $\beta_t(a)=e^{ith}\wmult a\wmult e^{-ith}$ is a well-defined weak *-automorphism of $(\A,\Ao)$ since $e^{ith}$, $e^{-ith}$ are bounded elements in $\A$ (see \cite{AT2}) and thus $(e^{ith}\wmult a)\wmult e^{-ith}=e^{ith}\wmult (a\wmult e^{-ith})$ for every $a\in\A$. Moreover, $\beta_t$ is a uniformly bounded norm continuous group of weak *-automorphisms. The infinitesimal generator is given by
$$\overline{\delta}_h(a):=\lim_{t\to0}\frac{\beta_t(a)-a}t=\lim_{t\to0}\frac{e^{ith}\wmult a\wmult e^{-ith}-a}{t}=i[h\wmult a-a\wmult h],$$
when \textit{$a$ is bounded}.
\end{example}

In the following Proposition, the *-semisimplicity is automatically given by assuming the existence of a representable and continuous functional with associated faithful *-representation (see \cite{adamo1,Brat2}).

\begin{proposition}\label{1} Let $(\A,\Ao)$ be a Banach quasi *-algebra with unit $\id$ and let $\delta$ be a weak *-derivation of $(\A,\Ao)$ such that $\mathcal{D}(\delta)=\Ao$. Suppose that there exists a representable and continuous functional $\omega$ with $\omega(\delta(x))=0$ for
	$x\in \Ao$ and let $(\H_\omega, \pi_\omega, \lambda_\omega)$ be the GNS-construction associated with $\omega$. Suppose that $\pi_\omega$ is a faithful *-representation of $(\A,\Ao)$. Then there exists an element $H = H^\dag$ of ${\mathcal L}\ad(\lambda_\omega(\Ao))$ such that
	$$\pi_\omega(\delta(x)) =-i[H,\pi_\omega(x)], \quad \forall x\in \Ao$$
	and $\delta$ is closable.
\end{proposition}

\section{2nd Application: Tensor product of Hilbert quasi *-algebras}\label{Sec4}
In this section, we recall the construction of the tensor product Hilbert quasi *-algebra of two given Hilbert quasi *-algebras $(\H_1,\Ao)$ and $(\H_2,\Bo)$, giving some results about the relationship between the representability properties for the tensor product and those for the factors. For further reading on the algebraic and topological tensor product, see \cite{cha,deflo,fra,hel2,lau,mal}, for the tensor product Hilbert quasi *-algebras refer to \cite{adamo,AF}.

For convenience, we assume that $(\H_1,\Ao)$ and $(\H_2,\Bo)$ are unital Hilbert quasi *-algebras. 

The algebraic tensor product $\Ao\otimes\Bo$ is the tensor product *-algebra of $\Ao$ and $\Bo$, it is endowed with the canonical multiplication and involution and it is considered as a subspace of the tensor product vector space $\H_1\otimes\H_2$.

If $\H_1$ and $\H_2$ are endowed with the inner product $\ip{\cdot}{\cdot}_1$ and $\ip{\cdot}{\cdot}_2$ respectively, then $\Ao\otimes\Bo$ satisfies the requirements of Definition \ref{Hilb_qalg} if we endow it with the following well-defined inner product
\begin{equation}\label{in_p}\ip{z}{z'}:=\sum_{i=1}^n\sum_{j=1}^m\ip{x_i}{x'_j}_1\ip{y_i}{y'_j}_2,\quad\forall z,z'\in\Ao\otimes\Bo,
\end{equation}
where $z=\sum_{i=1}^nx_i\otimes y_i$ and $z'=\sum_{j=1}^mx'_j\otimes y'_j$ (see \cite{mal,murphy,tak}). Then, the completion of $\Ao\otimes\Bo$ with respect to the norm $\|\cdot\|_h$ induced by the inner product in \eqref{in_p} is a Hilbert quasi *-algebra. Since $\Ao$, $\Bo$ are respectively dense in $\H_1$, $\H_2$ and $\|\cdot\|_h$ is a {\it cross-norm}, i.e. $\|x\otimes y\|_h=\|x\|_1\|y\|_2$ for all $x\otimes y\in\Ao\otimes\Bo$, the tensor product *-algebra $\Ao\otimes\Bo$ is $\|\cdot\|_h$-dense in $\H_1\otimes_h\H_2$. We conclude that
$$\Ao\widehat{\otimes}^h\Bo\equiv\H_1\widehat{\otimes}^h\H_2$$
and $(\H_1\widehat{\otimes}^h\H_2,\Ao\otimes\Bo)$ is a Hilbert quasi *-algebra, see \cite{adamo,AF}.

For the tensor product Hilbert quasi *-algebra $(\H_1\widehat{\otimes}^h\H_2,\Ao\otimes\Bo)$ we can apply all the known results for Banach quasi *-algebras about representability presented in Section \ref{Sec2}, see also \cite{AT2,AT}.
In particular, we know that $(\H_1\widehat{\otimes}^h\H_2,\Ao\otimes\Bo)$ is always a *-semisimple and fully representable Hilbert quasi *-algebra, applying Theorem \ref{thm_fullrep_semis} and \ref{Hrepc}. 

Employing Theorem \ref{Hrepc} and Lemma 4.1 in \cite{adamo}, we can show the following result, of which the proof is given below after the proof of Theorem \ref{sesq_tens_prod}.

\begin{theorem}\cite{adamo}\label{repr_tens_prod}
	Let $(\H_1\widehat{\otimes}^h\H_2,\Ao\otimes\Bo)$ be the tensor product Hilbert quasi *-algebra of $(\H_1,\Ao)$ and $(\H_2,\Bo)$. Then, if $\omega_1$, $\omega_2$ are representable and continuous functionals on $\H_1$ and $\H_2$ respectively, then $\omega_1\otimes\omega_2$ extends to a representable and continuous functional $\Omega$ on the tensor product Hilbert quasi *-algebra $\H_1\widehat{\otimes}^h\H_2$.
\end{theorem}

We now look at what happens to the sesquilinear forms in $\SSHone$ and $\SSHtwo$.

\begin{theorem}\label{sesq_tens_prod}
	Let $(\H_1\widehat{\otimes}^h\H_2,\Ao\otimes\Bo)$ be the tensor product Hilbert quasi *-algebra of $(\H_1,\Ao)$ and $(\H_2,\Bo)$. Let $\phi_1\in\SSHone$ and $\phi_2\in\SSHtwo$. Then, $\phi_1\widehat{\otimes}\phi_2\in\mathcal{S}_{\Ao\otimes\Bo}(\H_1\widehat{\otimes}^h\H_2)$, i.e., it satisfies the conditions (i), (ii) and (iii) after the Definition \ref{fully_rep}. 
\end{theorem}
\begin{proof}
Let $\phi_1\in\SSHone$ and $\phi_2\in\SSHtwo$. These sesquilinear forms are continuous, thus by the representation theorem for bounded sesquilinear forms over a Hilbert space, there exist unique bounded operators $T_1:\H_1\to\H_1$ and $T_2:\H_2\to\H_2$ such that
$$\phi_1(\xi,\xi')=\ip{\xi}{T_1\xi'}\quad\text{and}\quad\phi_2(\eta,\eta')=\ip{\eta}{T_2\eta'},$$
for all $\xi,\xi'\in\H_1$, $\eta,\eta'\in\H_2$, and $\|T_i\|=\|\phi_i\|\leq1$ for $i=1,2$. Moreover, $T_1$ and $T_2$ are positive operators on $\H_1$ and $\H_2$ respectively.

On the pre-completion $\H_1\otimes^h\H_2$, define the tensor product of $\phi_1$ and $\phi_2$ as
$$\phi_1\otimes\phi_2(\zeta,\zeta'):=\sum_{i,j=1}^n\phi_1(\xi_i,\xi'_i)\phi_2(\eta_j,\eta'_j)$$
for all $\zeta=\sum_{i=1}^n\xi_i\otimes\eta_i$, $\zeta'=\sum_{j=1}^n\xi'_j\otimes\eta'_j$ in $\H_1\otimes^h\H_2$.

By \cite{hel2}, it is known that $T_1\otimes T_2$ extends to a bounded operator $T_1\widehat{\otimes}T_2$ on the completion $\H_1\widehat{\otimes}^h\H_2$ such that 
$\|T_1\widehat{\otimes} T_2\|\leq\|T_1\|\|T_2\|=1$. 

We show that $\phi_1\otimes\phi_2$ is represented by $T_1\otimes T_2$ on $\H_1\otimes^h\H_2$. Thus, $\phi_1\otimes\phi_2$ is continuous and can be extended to $\H_1\widehat{\otimes}^h\H_2$. Its extension will be denoted by $\phi_1\widehat{\otimes}\phi_2$ and it corresponds to the bounded operator $T_1\widehat{\otimes}T_2$. By the definition of $\phi_1\otimes\phi_2$, we have
\begin{align*}
\phi_1\otimes\phi_2(\zeta,\zeta')&=\sum_{i,j=1}^n\phi_1(\xi_i,\xi'_i)\phi_2(\eta_j,\eta'_j)\\
&=\sum_{i,j=1}^n\ip{\xi_i}{T_1\xi'_i}\ip{\xi}{T_1\xi'}\\
&=\ip{\zeta}{(T_1\otimes T_2)\zeta'}.
\end{align*}

To conclude the proof, we show that $\phi_1\widehat{\otimes}\phi_2$ is in $\mathcal{S}_{\Ao\otimes\Bo}(\H_1\otimes^h\H_2)$. Indeed, $\phi_1\widehat{\otimes}\phi_2$ is a positive sesquilinear form, since $T_1\widehat{\otimes}T_2$ is a positive operator as a tensor product of positive operators on a Hilbert space. Thus, the condition (i) is verified.

The condition (ii) can be easily verified using the corresponding properties of $\phi_1$ and $\phi_2$. For (iii), we know that $\|\phi_1\widehat{\otimes}\phi_2\|=\|T_1\widehat{\otimes}T_2\|\leq\|T_1\|\|T_2\|\leq1$. Hence, all the conditions for $\phi_1\widehat{\otimes}\phi_2$ to belong to $\mathcal{S}_{\Ao\otimes\Bo}(\H_1\widehat{\otimes}^h\H_2)$ are verified.
\end{proof}

Using the same argument as in Theorem \ref{thm_fullrep_semis} (refer to \cite{AT} for a complete proof), the sesquilinear form associated with a representable functional in a Hilbert quasi *-algebra $(\H,\Ao)$ with unit $\id$ is bounded and in $\mathcal{Q}_{\Ao}(\mathcal{H})$. Hence, it is represented by a bounded operator $S_{\omega}$ such that
$$\omega(\xi)=\overline{\varphi}_{\omega}(\xi,\id)=\ip{\xi}{S_\omega \id},\quad \xi\in\H.$$
We want to show that $S_\omega\id$ is a weakly positive and bounded element. Indeed, for all $x\in\Ao$, we have
$$\ip{x}{R_{S_\omega\id}x}=\ip{x}{x(S_\omega\id)}=\ip{x^*x}{S_\omega\id}=\omega(x^*x)\geq0.,$$
thus $S_\omega\id$ is weakly positive. Moreover, $S_\omega\id$ is bounded by the condition (L.3) of representability in Definition \ref{fun_repr}. Indeed, (L.3) means that for every $\xi\in\H$, there exists $\gamma_{\xi}>0$ such that
$$\left|\omega(\xi^*x)\right|\leq\gamma_{\xi}\omega(x^*x)^{\frac12},$$
for all $x\in\Ao$. By Proposition \ref{prop1}, $\overline{\varphi}_{\omega}$ is everywhere defined and bounded, hence
$$\left|\omega(\xi^*x)\right|=\left|\overline{\varphi}_{\omega}(x,\xi)\right|\leq c_{\xi}\|x\|,$$
for some positive constant $c_{\xi}$. Hence, by the Riesz representation theorem, there exists $\chi\in\H$ such that $\omega(\xi^*x)=\ip{x}{\chi}$ for all $x\in\Ao$. Therefore, the weak product (in the sense of \cite[Definition 4.4]{AT}) $\xi\wmult S_\omega\id$ is well-defined for all $\xi\in\H$. Indeed, for $x,y\in\Ao$, we have
$$\ip{\xi^*x}{S_\omega\id y}=\ip{\xi^*xy^*}{S_\omega\id}=\ip{xy^*}{\chi}=\ip{x}{\chi y}.$$
A similar argument shows that $S_{\omega}\id\wmult\xi$ is well-defined for all $\xi\in\H$.
Recall that an element $\xi\in\H$ is bounded if, and only if, $R_w(\xi)=L_w(\xi)=\H$, where $R_w(\xi)$ (resp. $L_w(\xi)$) is the space of universal right (resp. left) weak multipliers of $\xi$ (see \cite[Proposition 4.10]{AT}). Then, $S_\omega\id$ is a bounded element. 

By Theorem \ref{Hrepc}, we have that $S_\omega\id$ is the weakly positive and bounded element in $\H$ corresponding to the reprensentable and continuous functional $\omega$.
\smallskip

For what we just argued, we can give an alternative proof of Theorem \ref{repr_tens_prod}. The proof tells us explicitly what the elements in $\H_1\widehat{\otimes}\H_2$ are corresponding to representable and continuous functionals of the form $\omega_1\widehat{\otimes}\omega_2$ for $\omega_1\in\mathcal{R}_c(\H_1,\Ao)$, $\omega_2\in\mathcal{R}_c(\H_2,\Bo)$.

\begin{proof}[Proof of Theorem \ref{repr_tens_prod}] Let $\omega_1$, $\omega_2$ be representable and continuous functionals on $\H_1$, $\H_2$ respectively. Then, for what we discussed above, $\omega_1(\xi)=\ip{\xi}{S_{\omega_1}\id_{\H_1}}$ for all $\xi\in\H_1$ and $\omega_2(\eta)=\ip{\eta}{S_{\omega_2}\id_{\H_2}}$ for $\eta\in\H_2$. 
	
Looking at their tensor product on $\H_1\otimes\H_2$, we have
$$\omega_1\otimes\omega(\zeta)=\sum_{i=1}^n\ip{\xi_i\otimes\eta_i}{S_{\omega_1}\id_{\H_1}\otimes S_{\omega_2}\id_{\H_2}},$$
for every $\zeta=\sum_{i=1}^n\xi_i\otimes\eta_i$ in $\H_1\otimes\H_2$.


Since $\overline{\varphi}_{\omega_1}$ and $\overline{\varphi}_{\omega_2}$ are in $\mathcal{Q}_{\Ao}(\H_1)$ and $\mathcal{Q}_{\Bo}(\H_2)$ respectively, with the same argument of Theorem \ref{sesq_tens_prod}, $\overline{\varphi}_{\omega_1}\widehat{\otimes}\overline{\varphi}_{\omega_2}$ is continuous and belongs to $\mathcal{Q}_{\Ao\otimes\Bo}(\H_1\widehat{\otimes}\H_2)$. Furthermore, it is represented by $S_{\omega_1}\widehat{\otimes}S_{\omega_2}$. Hence
$$S_{\omega_1}\otimes S_{\omega_2}(\id_{\H_1}\otimes\id_{\H_2})=S_{\omega_1}\id_{\H_1}\otimes S_{\omega_2}\id_{\H_2}$$
is weakly bounded and positive. Hence, by Theorem \ref{Hrepc}, $\omega_1\otimes\omega_2$ is representable and continuous. Furthermore, the continuous extension of $\omega_1\otimes\omega_2$ is representable. Indeed, let us denote this extension as $\Omega$, then
\begin{align*}
\Omega(\psi)&=\lim_{n\to+\infty}\omega_1\otimes\omega_2\left(z_n\right)\\
&=\lim_{n\to+\infty}\ip{z_n}{S_{\omega_1}\otimes S_{\omega_2}(\id_{\H_1}\otimes\id_{\H_2})}\\
&=\ip{\psi}{S_{\omega_1}\otimes S_{\omega_2}(\id_{\H_1}\otimes\id_{\H_2})},
\end{align*}
where $\{z_n\}$ is a sequence of elements in $\Ao\otimes\Bo$ $\|\cdot\|_h$-converging to $\psi\in\H_1\widehat{\otimes}\H_2$.
\end{proof}

It would be of interest to look at the same questions in the more general framework of Banach quasi *-algebras. This is work in progress with Maria Fragouloupoulou. In this case, we need further assumptions on the considered cross-norm to get some of the properties that we showed in this work, see \cite{AF}.

\smallskip

{\bf{Acknowledgment:} }
The author is grateful to the organizers of the International Workshop on Operator Theory and its Applications 2019, especially to the Organizers of the section entitled "Linear Operators and Function Spaces", for this interesting and delightful conference and the Instituto Superior Tecnico of Lisbon for its hospitality. The author was financially supported by the ERC Advanced Grant QUEST ``Quantum Algebraic Structures and Models". 

The author wishes to thank the anonymous referees for their useful suggestions that improved the presentation of this manuscript.

\end{document}